\documentclass{birkjour}
\usepackage{microtype}
\usepackage{amsfonts, tikz}
\usepackage{amssymb,amsmath}
\usepackage{amsthm,mathtools}
\usepackage{graphicx}
\usepackage{amscd}
\newtheorem{thm}{Theorem}[section]
\newtheorem{cor}[thm]{Corollary}
\newtheorem{lem}[thm]{Lemma}
\newtheorem{prop}[thm]{Proposition}
\theoremstyle{definition}
\newtheorem{defn}[thm]{Definition}
\theoremstyle{remark}
\newtheorem{rem}[thm]{Remark}
\newtheorem{ex}[thm]{Example}
\numberwithin{equation}{section}

\renewcommand{\labelenumi}{(\roman{enumi})}

\newcommand{\Hom}{\mathrm{Hom}}
\newcommand{\id}{\mathrm{id}}

\begin{document}
		\title[On $\ell$-open $C^*$ algebras and $\ell$-closed $C^*$-algebras]{On $\ell$-open and $\ell$-closed $C^*$-algebras}
		%----------Author 1
		\author{Dolapo Oyetunbi}
		
		\address{Department of Mathematics and Statistics\\
		University of Ottawa\\
		Canada}
		\email{Doyet074@uottawa.ca}
		%----------Author 2
		\author{Aaron Tikuisis}
	\address{Department of Mathematics and Statistics\\
		University of Ottawa\\
		Canada}
	\email{Aaron.tikuisis@uottawa.ca}
		%-------classification, keywords, date
%		\subjclass{Primary 47H04, 47H09, 47H10}
		
%		\keywords{Asymptotically regular mapping, common stationary point, multivalued mapping, metric space, orbitally continuous }
	%	\thanks{*Corresponding Author}
		%\date{January 1, 2004}
		%%% ----------------------------------------------------------------------
		\begin{abstract} In this paper, we characterize $\ell$-open and $\ell$-closed $C^*$-algebras and deduce that $\ell$-open $C^*$-algebras are $\ell$-closed, as conjectured by Blackadar. 
Moreover, we show that a commutative unital $C^*$-algebra is $\ell$-open if and only if it is semiprojective.
		\end{abstract}
		\maketitle
	\section{Introduction}
Lifting properties of $C^*$-algebras and their $^*$-ho\-mo\-mor\-phisms have been well-studied for some time with prominent connections to notions of stability; see \cite{Bla85,EilersShulmanSorensen,Lor93,Lor97} for example.
They play an important role in modern $C^*$-algebra theory including the Elliott classification program (\cite{Gong+Lin+Niu,Dad09,OP12}, for example).
In connection to a non-commutative generalization of Borsuk's homotopy extension theorem, Blackadar \cite{Bla16} defined natural classes of $C^*$-algebras in terms of lifting properties, called $\ell$-open and $\ell$-closed $C^*$-algebras.
A $C^*$-algebras is $\ell$-open if the liftable maps from the $C^*$-algebra to any quotient $C^*$-algebra is a point-norm open set, and $\ell$-closedness is defined similarly (precise definitions can be found in Section~\ref{sec:Prelim}).

While these notions are first formalized only recently by Blackadar, their study traces back at least to the celebrated work of Brown, Douglas, and Fillmore: in \cite{BDF2}, they seek conditions on a space $X$ that ensure the set of liftable maps from $C(X)$ to the Calkin algebra is closed.
It is open whether $C(\mathbb D)$ is $\ell$-closed, and a positive answer would settle an open question on page 119 of \cite{BDF1}.
More recently, Enders and Shulman further studied when the set of liftable maps from $C(X)$ to the Calkin algebra is closed, including a sufficient condition when $\dim(X)\leq 2$ and a full characterization when $\dim(X)\leq 1$ \cite{Enders+Shulman}.

In this paper, we prove the following characterizations of being $\ell$-open and $\ell$-closed:

\begin{thm}[{see Theorem~\ref{thm:Main}}]
Let $A$ be a $C^*$-algebra.
The following are equivalent:
\begin{enumerate}
\item $A$ is $\ell$-open.
\item For every $C^*$-algebra $B$ and ideal $I\subseteq B$, the natural map $\Hom(A,B) \to \Hom(A,B/I)$ is open.
\item $A$ satisfies the Homotopy Lifting Theorem (a noncommutative analog of the Borsuk Homotopy Extension Theorem), and $\Hom(A,B)$ is locally path-connected for every $C^*$-algebra $B$.
\end{enumerate}
\end{thm}

Condition (ii) can be strengthened to uniform openness (see Theorem~\ref{thm:Main}).

\begin{thm}[{see Theorem~\ref{thm:MainClosed}}]
Let $A$ be a separable $C^*$-algebra.
Then $A$ is $\ell$-closed if and only if for every $C^*$-algebra $B$ and ideal $I\subseteq B$, the natural map $\Hom(A,B) \to \Hom(A,B/I)$ is uniformly relatively open.
\end{thm}

As a consequence, we confirm a conjecture of Blackadar from \cite{Bla16}, that $\ell$-open $C^*$-algebras are $\ell$-closed.

Additionally, we prove that a unital commutative $C^*$-algebra is semiprojective if and only if it is $\ell$-open, confirming another conjecture from \cite[Page 299]{Bla16}.

%So, for any C*-algebras B with ideals I, an l-open C*-algebra A ensures that every $^*$-ho\-mo\-mor\-phism $\phi: A \rightarrow B/I$ that is sufficiently close to a liftable $^*$-ho\-mo\-mor\-phism $\psi: A \rightarrow B/I$ is liftable,  while an l-closed C*-algebra A ensures that point-norm limit of a sequence of liftable $^*$-ho\-mo\-mor\-phisms $\phi_n: A \rightarrow B/I$ is liftable. 
	\section{Preliminaries}
\label{sec:Prelim}
Let $A$ and $B$ be $C^*$-algebras and let $I$ an ideal in $B$ (by which we mean a closed, two-sided ideal).
We write $\pi_I : B \rightarrow B/I$ be the quotient map. 
Recall that a $^*$-ho\-mo\-mor\-phism $\phi: A \rightarrow B/I$ is \textit{liftable} if there exists a $^*$-ho\-mo\-mor\-phism $\overline{\phi} : A \rightarrow B$ such that $ \phi = \pi_I \circ \overline{\phi}$:

\tikzset{every picture/.style={line width=0.75pt}} %set default line width to 0.75pt        

\begin{tikzpicture}[x=0.75pt,y=0.75pt,yscale=-1,xscale=1]
uncomment if require: \path (220,100); %set diagram left start at 0, and has height of 300

%Straight Lines [id:da7569584418201232] 
\draw    (339.11,158.3) -- (470.11,159.28) ;
\draw [shift={(472.11,159.3)}, rotate = 180.43] [color={rgb, 255:red, 0; green, 0; blue, 0 }  ][line width=0.75]    (10.93,-3.29) .. controls (6.95,-1.4) and (3.31,-0.3) .. (0,0) .. controls (3.31,0.3) and (6.95,1.4) .. (10.93,3.29)   ;
%Straight Lines [id:da3473253869763919] 
\draw    (486.11,88.3) -- (486.11,146.3) ;
\draw [shift={(486.11,148.3)}, rotate = 270] [color={rgb, 255:red, 0; green, 0; blue, 0 }  ][line width=0.75]    (10.93,-3.29) .. controls (6.95,-1.4) and (3.31,-0.3) .. (0,0) .. controls (3.31,0.3) and (6.95,1.4) .. (10.93,3.29)   ;
%Straight Lines [id:da039707731159484894] 
\draw  [dash pattern={on 4.5pt off 4.5pt}]  (341.11,152.3) -- (430.06,105.2) -- (475.34,81.23) ;
\draw [shift={(477.11,80.3)}, rotate = 152.1] [color={rgb, 255:red, 0; green, 0; blue, 0 }  ][line width=0.75]    (10.93,-3.29) .. controls (6.95,-1.4) and (3.31,-0.3) .. (0,0) .. controls (3.31,0.3) and (6.95,1.4) .. (10.93,3.29)   ;

% Text Node
\draw (325,150) node [anchor=north west][inner sep=0.75pt]   [align=left] {A};
% Text Node
\draw (473,150.4) node [anchor=north west][inner sep=0.75pt]    {$B/I$};
% Text Node
\draw (482,70) node [anchor=north west][inner sep=0.75pt]   [align=left] {B};
% Text Node
\draw (392,161.4) node [anchor=north west][inner sep=0.75pt]    {$\phi $};
% Text Node
\draw (491,107.4) node [anchor=north west][inner sep=0.75pt]    {$\pi _{I}$};
% Text Node
\draw (373,83.4) node [anchor=north west][inner sep=0.75pt]    {$\exists \ \ \overline{\phi }$};

\end{tikzpicture}

We denote the space of $^*$-ho\-mo\-mor\-phisms from $A$ to $B$ endowed with the point-norm topology by  $\Hom (A, B)$ and the subspace of unital $^*$-ho\-mo\-mor\-phisms by $\Hom_1 (A, B)$ (if $A$ and $B$ are unital). 
For $\phi \in \Hom(A,B)$, a neighbourhood base of $\phi$ is made up of sets
\begin{equation}
\label{eq:UBdef}
 U_B(\phi;\mathcal F,\epsilon) \coloneqq \{\psi \in\Hom(A,B): \|\psi(a)-\phi(a)\|<\epsilon\ \forall a \in \mathcal F\},
\end{equation}
ranging over all finite sets $\mathcal F \subset A$ and all positive real numbers $\epsilon>0$.
This gives a uniform structure to $\Hom(A,B)$.
In fact, the sets of this neighbourhood base are parametrized independently of $B$, giving a uniform structure to all of $\Hom(A,B)$ at once.
(One would like to put a uniform structure on the disjoint union of $\Hom(A,B)$ ranging over all $C^*$-algebras $B$, except that this is not a well-founded set.
One can put a uniform structure on $\coprod_{B\in\mathcal B}  \Hom(A,B)$, for any set $\mathcal B$ of $C^*$-algebras.)

The set of liftable $^*$-ho\-mo\-mor\-phisms $A \to B/I$ is
\begin{equation} \Hom(A,B,I) \coloneqq \pi_I \circ \Hom(A,B). \end{equation}

%When $A$ and $B$ are separable, $\Hom (A, B)$ is metrizable and separable.
%\marginpar{Do we really need this?} Indeed, if $\mathcal{G}= \{a_1 , a_2 , \ldots \}$ is a generating set of $A$ satisfying $\lim\limits_{n\rightarrow \infty} \Vert a_n \Vert = 0$, we can define a metric on $\Hom(A, B)$ as follows: for $\phi, \psi \in  \Hom(A, B)$,
%$$ d(\phi, \psi) = \sup_j \Vert \phi(a_j) - \psi(a_j) \Vert, $$
%and this metric induces the point-norm topology.
%If $A = C(Y_1)$ and $B= C(Y_2)$, where $Y_1$ and $Y_2$ are compact metrizable spaces, then $\Hom_{1} (A, B)$ is homeomorphic to the set of continuous function from $Y_2$ to $Y_1$ endowed with the topology of uniform convergence (by Gelfand duality).

%Let $\mathcal{D}$ be a category of $C^*$-algebras closed under quotients (i.e., if $B\in \mathcal{D}$, and $I$ is an ideal in $B$, then $B/I \in \mathcal{D}$ and the quotient map $\pi_I$ is a $\mathcal{D}$-morphism). If $A, B\in \mathcal{D}$, we denote $\Hom_{\mathcal{D}} (A,B)$ as the space of $\mathcal{D}$-morphisms from $A$ to $B$ equipped with the point-norm topology inherited from $\Hom(A, B)$ and $\Hom_{\mathcal{D}} (A,B, I)$ as the space of liftable $\mathcal{D}$-morphisms from $A$ to $B/I$. \marginpar{Do we really use this?}

The following is due to Blackadar \cite[Definition 6.1]{Bla16}.

\begin{defn}
%Let $\mathcal{D}$ be a category closed under quotients and $A\in \mathcal{D}$
Let $A$ be a $C^*$-algebra
\begin{enumerate}
	\item $A$ is $\ell$-\textit{open} %(in $\mathcal{D}$), 
if for any $C^*$-algebra $B$ and ideal $I$ of $B$, 
%every pair $(B,I) \in \mathcal{D}$, 
the set $\Hom(A, B, I)$ is open in $\Hom(A, B/I)$.
	\item $A$ is $\ell$-\textit{closed} %(in $\mathcal{D}$), 
if for any $C^*$-algebra $B$ and ideal $I$ of $B$, 
%every pair $(B,I) \in \mathcal{D}$, 
the set $\Hom(A, B, I)$ is closed in $\Hom(A, B/I)$.
\end{enumerate}
\end{defn}
%We simply say $A$ is $\ell$-\textit{open} (or $\ell$-\textit{closed}) if $\mathcal{D}$ is the category of all separable $C^*$-algebras. Blackadar \cite{Bla12} observed that, via Gelfand duality, if $\mathcal{D}$ is the category of separable unital commutative $C^*$-algebras and $A = C(X)$, then $A$ is $\ell$-open (or $\ell$-closed) in $\mathcal{D}$ if and only if $X$ is $e$-open (or $e$-closed). 
%Moreover, he showed that an $e$-open space is locally contractible (\cite{Bla12}, Corollary.
\begin{defn}
%Let $A$ be a $C^*$-algebra. $I_1 \triangleleft I_2 \triangleleft \cdots \triangleleft I$ an increasing sequence of closed ideals of a $C^*$-algebra $B$ with $I = \overline{\cup I_n}$. 

Recall that a $^*$-ho\-mo\-mor\-phism $\phi: A\rightarrow C$ is \textit{(weakly) semiprojective} if for any $C^*$-algebra $B$, any increasing sequence $I_1 \triangleleft I_2 \triangleleft \cdots \triangleleft B$ of ideals in $B$, and any $^*$-ho\-mo\-mor\-phism $\psi : C \rightarrow B/\overline{\bigcup_n I_n}$ (and finite set $F\subset A$, $\epsilon >0$), there is an $n$ and a $^*$-ho\-mo\-mor\-phism $\overline{\psi} : A \rightarrow B/I_n$ such that $\psi \circ \phi = \pi_I \circ \overline{\psi}$ (resp.\ $\Vert \psi  \circ \phi (x) - \pi_I \circ \overline{\psi}(x) \Vert < \epsilon$ for all $x\in F$), where $\pi_I : B/I_n \rightarrow B/I$ is the quotient map.

\tikzset{every picture/.style={line width=0.75pt}} %set default line width to 0.75pt        

\begin{tikzpicture}[x=0.45pt,y=0.45pt,yscale=-1,xscale=1]
	uncomment if require: \path (-50,150); %set diagram left start at 0, and has height of 300
	
	%Straight Lines [id:da11318750423630064] 
	\draw    (232,169) -- (366,169) ;
	\draw [shift={(369,169)}, rotate = 180] [fill={rgb, 255:red, 0; green, 0; blue, 0 }  ][line width=0.08]  [draw opacity=0] (8.93,-4.29) -- (0,0) -- (8.93,4.29) -- cycle    ;
	%Straight Lines [id:da5405018627207867] 
	\draw    (393,98) -- (391.11,151) ;
	\draw [shift={(391,154)}, rotate = 272.05] [fill={rgb, 255:red, 0; green, 0; blue, 0 }  ][line width=0.08]  [draw opacity=0] (8.93,-4.29) -- (0,0) -- (8.93,4.29) -- cycle    ;
	%Straight Lines [id:da8223978598533583] 
	\draw  [dash pattern={on 4.5pt off 4.5pt}]  (112.11,162.02) -- (269.11,119.02) -- (378.2,91.75) ;
	\draw [shift={(381.11,91.02)}, rotate = 165.96] [fill={rgb, 255:red, 0; green, 0; blue, 0 }  ][line width=0.08]  [draw opacity=0] (8.93,-4.29) -- (0,0) -- (8.93,4.29) -- cycle    ;
	%Straight Lines [id:da5156999606779575] 
	\draw    (110.11,170.02) -- (203.11,169.04) ;
	\draw [shift={(205.11,169.02)}, rotate = 179.4] [color={rgb, 255:red, 0; green, 0; blue, 0 }  ][line width=0.75]    (10.93,-3.29) .. controls (6.95,-1.4) and (3.31,-0.3) .. (0,0) .. controls (3.31,0.3) and (6.95,1.4) .. (10.93,3.29)   ;
	
	% Text Node
	\draw (213,158.4) node [anchor=north west][inner sep=0.75pt]    {$C$};
	% Text Node
	\draw (374,157.4) node [anchor=north west][inner sep=0.75pt]    {$B/\overline{\bigcup{}_{n} \ I_{n}} \ $};
	% Text Node
	\draw (385,74.4) node [anchor=north west][inner sep=0.75pt]    {$B/I_{n}$};
	% Text Node
	\draw (145,175.4) node [anchor=north west][inner sep=0.75pt]    {$\phi $};
	% Text Node
	\draw (404,114.4) node [anchor=north west][inner sep=0.75pt]    {$\pi _{_{I}}$};
	% Text Node
	\draw (229,94.4) node [anchor=north west][inner sep=0.75pt]    {$\overline{\psi }$};
	% Text Node
	\draw (90,161.4) node [anchor=north west][inner sep=0.75pt]    {$A$};
	% Text Node
	\draw (294,177.4) node [anchor=north west][inner sep=0.75pt]    {$\psi $};

\end{tikzpicture}
%\begin{figure}[h!]
%	\centering
%	\includegraphics[width=0.7\linewidth]{semiprojective}
%	\caption{}
%	\label{fig:semiprojective}
%\end{figure}
	
	%\vspace*{1em}
$A$ is \textit{(weakly) semiprojective} if the identity $^*$-ho\-mo\-mor\-phism is (weakly) semiprojective. Some examples of semiprojective $C^*$-algebras are finite dimensional $C^*$-algebras, the universal $C^*$-algebras generated by $n$ unitaries, $C^*(\mathbb{F}_n)$,  and $\{f\in C(S^1 , \mathbf{M}_n) : f(1) \,\text{is scalar}\}$ (see \cite{Lor97}).
	%\item  The Cuntz algebras $O_n$, $n\leq \infty$
\end{defn}

\begin{ex}\cite[Corollary 6.2]{Bla16} All semiprojective $C^*$-algebras are both $\ell$-open and $\ell$-closed $C^*$-algebras.
\end{ex}

By slight abuse of notation, if $L\subseteq K \subseteq B$ are ideals, then we also use $\pi_K$ to denote the quotient map from $B/L$ to $B/K$.

We recall the following general Chinese remainder theorem for $C^*$-algebras:
\begin{lem}[\cite{Bla16}, Proposition 2.1]
\label{lem:CRT}
	Let $B$ be a $C^*$-algebra, and $I$ and $J$ ideals in $B$. Then $B/(I \cap J)$ is isomorphic to the fibred product
	$ \{(x,y) \in x\in B/I \oplus y \in B/J: \pi_{I+J}(x)= \pi_{I+J}(y) \}$
via the map $a \rightarrow (\pi_I(a),\pi_J(a))$.
\end{lem}
	
\section{Properties and characterization of $\ell$-open $C^*$-algebras}

The following shows that if $A$ is $\ell$-open then the quotient map $\Hom(A,B) \to \Hom(A,B/I)$ is always open.
In fact, it shows that this openness is uniform, as the relationship between $(\mathcal G,\delta)$ and $(\mathcal F,\epsilon)$ in the statement below does not depend on the $C^*$-algebra $B$, the ideal $I$, nor any of the $^*$-ho\-mo\-mor\-phisms under consideration.
The conclusion of the following theorem is (in the separable case) a reformulation of the conclusion of \cite[Theorem 4.1]{Bla16}; the ideas in the proof are similar, but work is needed to allow $\ell$-openness instead of semiprojectivity as the hypothesis.

	\begin{thm}
\label{thm:lopen}
		Let $A$ be an $\ell$-open $C^*$-algebra.
		 %generated by a finite or countable set $\mathcal{G} = \{a_1, a_2 , \ldots \}$ with $\lim\limits_{n \rightarrow \infty } \Vert a_n \Vert =0$ if $\mathcal{G}$ is infinite. 
		Then for any $\epsilon > 0$ and any finite set $\mathcal{F} \subset A$, there is a $\delta >0$ and a finite set $\mathcal{G} \subset A$ such that whenever $B$ is a $C^*$-algebra, $I$ is an ideal of $B$, $\gamma$ and $\varphi$ are $^*$-ho\-mo\-mor\-phisms from $A$ to $B/I$ with $\Vert \gamma (u) - \varphi(u) \Vert < \delta$ for all $u\in \mathcal{G}$ and such that $\gamma$ lifts to a $^*$-ho\-mo\-mor\-phism $\overline{\gamma}: A \rightarrow B$, then $\varphi$ also lifts to a $^*$-ho\-mo\-mor\-phism $\overline{\varphi}: A \rightarrow B$ with $\Vert \overline{\gamma}(v) - \overline{\varphi}(v) \Vert < \epsilon$ for all $v \in \mathcal{F}$.
In other words, in the notation of \eqref{eq:UBdef},
\begin{equation} U_{B/I}(\gamma; \mathcal G, \delta) \subseteq \pi_I \circ U_B(\overline\gamma; \mathcal F, \epsilon).  \end{equation}
	\end{thm}

	\begin{proof}
Let $(\mathcal{G}_n)_{n\in \Lambda}$ be an increasing net of finite subsets of $A$ whose union is dense in $A$, and let $(\delta_n)_{n\in \Lambda}$ be a net (over the same index set) of positive numbers such that $\delta_n \to 0$.
	 Suppose that the conclusion of the theorem is false for a fixed $\epsilon >0$ and finite set $\mathcal{F}$. Then, there are $C^*$-algebras $B_n$ with ideals $I_n$ and $^*$-ho\-mo\-mor\-phisms $\gamma_n, \varphi_n : A \rightarrow B_n /I_n$ such that 
\begin{equation}
\label{eq:gammavarphiclose}
\Vert \gamma_n (u) - \varphi_n(u) \Vert < \delta_n
\end{equation}
 for all $u\in \mathcal{G}_n$, $\gamma_n$ lifts to $\overline{\gamma}_n: A \rightarrow B_n$, but no $\varphi_n$ lifts to $^*$-ho\-mo\-mor\-phism $\overline{\varphi}_n: A \rightarrow B_n$ with $\Vert \overline{\gamma}_n (v) - \overline{\varphi}_n(v) \Vert < \epsilon$ for all $v\in \mathcal{F}$.
		
		Let $B \coloneqq \prod \limits_{n\in\Lambda} B_n$, $I \coloneqq \prod \limits_{n\in\Lambda} I_n$, and $J \coloneqq \{(b_n)\in B: \lim_n \|b_n\| = 0\}$.
 Then $ B/I \cong \prod \limits_{n\in\Lambda} B_n / I_n$.
	Define $^*$-ho\-mo\-mor\-phisms $\overline{\gamma} \coloneqq (\overline\gamma_n)_{n\in\Lambda}: A \rightarrow B$ and $\varphi\coloneqq (\varphi_n)_{n\in\Lambda}: A \rightarrow B/I$.
Then \eqref{eq:gammavarphiclose} implies that $\lim\limits_{n} \Vert \gamma_n (x) - \varphi_n(x) \Vert = 0$ for all $x\in A$, and so $\pi_{I+J} \circ  \overline{\gamma} = \pi_{I+J} \circ \varphi$. 
		
%	Let $\phi = \pi_J \circ \overline{\phi} :A \rightarrow B/J$. 
	Using the general Chinese remainder theorem (Lemma \ref{lem:CRT}), there exists a $^*$-ho\-mo\-mor\-phism $\theta : A \rightarrow B/ (I\cap J)$ such that 
		\begin{equation}
\label{eq:thetadef}
			\pi_{J} \circ \overline{\gamma} = \pi_{J} \circ \theta\,\, \text{ and}\,\, \varphi = \pi_{I} \circ \theta 
		\end{equation}
Take a $^*$-linear lift $(\theta_n)_{n\in\Lambda}:A \to B$ of $\theta$ (which need not be a $^*$-ho\-mo\-mor\-phism), thus defining $\theta_n:A \to B_n$.
For $m\in\Lambda$, define $\alpha_m\coloneqq \pi_{I \cap J} \circ (\alpha_{m,n})_{n\in\Lambda}$, where
\begin{equation} \alpha_{m,n} \coloneqq  \begin{cases} \theta_n,\quad &n\geq m; \\ \overline\gamma_n,\quad &\text{otherwise}. \end{cases} \end{equation}
Since $\theta$ is a $^*$-ho\-mo\-mor\-phism, %$\lim\limits_{n\rightarrow \infty} \Vert \theta_n(x+y)- \theta_n (x)-\theta_n (y) \Vert = 0$
		$\|\theta_n(xy)- \theta_n (x)\theta_n (y)\| \to 0$ for all $x,y \in A$; from this it follows that $\alpha_{m,n}$ is also a $^*$-ho\-mo\-mor\-phism.

The first equation of \eqref{eq:thetadef} implies that $\lim_n \|\overline\gamma_n(x)-\theta_n(x)\| = 0$ for all $x \in A$, which in turn implies that 
\begin{equation} \|\alpha_m(x)-\pi_{I\cap J}(\overline\gamma(x))\| = \sup_{n\geq m} \|\overline\gamma_n(x) - \theta_n(x)\| \to 0 \end{equation}
for all $x \in A$.
Thus, $(\alpha_m)_m$ converges in the point-norm topology to the liftable $^*$-ho\-mo\-mor\-phism $\pi_{I\cap J} \circ \overline\gamma$, and since $A$ is $\ell$-open, it follows that $\alpha_m$ is liftable for some sufficiently large $m$.
Let $\beta=(\beta_n)_{n\in\Lambda}: A \rightarrow B$ be a lift of $\alpha_m$, where $\beta_n:A \to B_n$ is a $^*$-ho\-mo\-mor\-phism for each $n$.
The fact that $\beta$ is a lift amounts to
\begin{equation}
\label{eq:betalift}
 (\beta_n(x)-\alpha_{m,n}(x))_{n\in\Lambda} \in I \cap J,\quad \text{for all }x\in A.
\end{equation}
This implies first that $\lim_n \|\beta_n(x)-\theta_n(x)\| = 0$ for all $x \in A$, and combining this with the first equation of \eqref{eq:thetadef}, it follows that 
\begin{equation} \lim_n \|\beta_n(x)-\overline\gamma_n(x)\| = 0, \quad \text{for all }x\in A. \end{equation}
From \eqref{eq:betalift}, we also get that $\pi_{I_n} \circ \beta_n(x)-\pi_{I_n} \circ \theta_n$ for all $n \geq m$, and combining this with the second equation of \eqref{eq:thetadef}, we have that $\beta_n$ is a lift of $\varphi_n$ for $n \geq m$.
In summary, for sufficiently large $n$ we find that $\beta_n$ is a lift of $\varphi_n$ which is point-norm close to $\overline\gamma_n$, in contradiction to our initial assumption.
\end{proof}

	We now pick up some consequences, using ideas from of Blackadar \cite{Bla16}.
We add the proofs for completion.
The first tells us that when $A$ is $\ell$-open, $\Hom(A,B)$ is locally path-connected in a uniform way.

\begin{cor}[{cf.\ \cite[Corollary 4.2]{Bla16}}] \label{cor:PathConnected}
Let $A$ be an $\ell$-open $C^*$-algebra (or more generally, one that satisfies the conclusion of Theorem~\ref{thm:lopen}).
 For any $\epsilon > 0$ and any finite set $\mathcal{F} \subset A$, there is a $\delta >0$ and a finite set $\mathcal{G} \subset A$ such that whenever $B$ is a $C^*$-algebra, $\varphi_0$ and $\varphi_1$ are $^*$-ho\-mo\-mor\-phisms from $A$ to $B/I$ with $\Vert \varphi_0 (u) - \varphi_1(u) \Vert < \delta$ for all $u\in \mathcal{G}$, then there is a point-norm continuous path $(\varphi_t)_{t\in [0,1]}$ of $^*$-ho\-mo\-mor\-phisms from $A$ to $B$ connecting $\varphi_0$ and $\varphi_1$ with $\Vert \varphi_0 (v) - \varphi_t (v) \Vert < \epsilon $ for all $v\in \mathcal{F}$ and $t\in [0,1]$. In particular, $Hom (A, B)$ is locally path-connected for any $C^*$-algebra $B$. 
\end{cor}
\begin{proof} 
For any $\epsilon>0$ and finite set $\mathcal{F}$, choose $\delta>0$ and finite set $\mathcal{G}$ as in Theorem \ref{thm:lopen}. Let $D\coloneqq  C([0,1], B)$ and $I\coloneqq  C_0 ((0,1), B)$. Then $D/I \cong B \oplus B$. Define $^*$-ho\-mo\-mor\-phisms $ \gamma, \varphi : A \rightarrow D/I$ by $\gamma(x) \coloneqq (\varphi_0 (x), \varphi_0(x))$ and  $\varphi(x) \coloneqq (\varphi_0 (x), \varphi_1(x))$. Then $\gamma$ lifts to a $^*$-ho\-mo\-mor\-phism $\id_{C([0,1])} \otimes \varphi_0: A \rightarrow D$, and so these two maps satisfy the hypothesis of Theorem \ref{thm:lopen}. Hence the conclusion of Theorem \ref{thm:lopen} holds and there exists a $^*$-ho\-mo\-mor\-phism $\overline{\varphi}=(\overline\varphi_t)_{t\in[0,1]}: A \rightarrow D$ such that 
\begin{equation}
\label{eq:PathConnected1}
\Vert \overline{\gamma}(a) - \overline{\varphi}(a) \Vert < \epsilon\quad \text{for all }a\in \mathcal{F}.
\end{equation}
Then $\overline\varphi$ is a homotopy of $^*$-ho\-mo\-mor\-phisms $A \to B$ connecting $\varphi_0$ to $\varphi_1$, and \eqref{eq:PathConnected1} tells us that $\|\varphi_t(a)-\varphi_0(a)\|<\epsilon$ for all $a\in\mathcal F$, as required.
\end{proof}

\begin{ex}
	Consider the topologist' sine curve:
\begin{equation} X \coloneqq \{(x,y): y= \sin(\frac{\pi}{x}), 0< x\leq 1 \} \cup \{(0,y): -1\leq y\leq 1\}. \end{equation}
Then $\Hom(C(X),\mathbb C)=X$, which is not locally path-connected; therefore by the above corollary, $C(X)$ is not $\ell$-open.
\end{ex}
%\begin{cor}
%	Let $A$ be an $\ell$-open $C^*$-algebra. 
%	%generated by a finite or countable generated set $\mathcal{G} = \{a_1 , a_2, \ldots\}$ with $\lim \limits _{j\rightarrow \infty} \Vert a_j \Vert =0$ if $\mathcal{G}$ is infinite. 
%	Then, there is a $\delta >0$ and a finite set $\mathcal{G}$ such that whenever $B$ is a $C^*$-algebra, $\varphi_0$ and $\varphi_1$ $^*$-ho\-mo\-mor\-phisms from $A$ to $B$ with $\Vert \varphi_0 (u) - \varphi_1 (u) \Vert < \delta$ for all $u\in \mathcal{G}$, then $\varphi_0 $ and $\varphi_1$ are homotopic.
%\end{cor}
%\begin{proof}
%	Fix any $\epsilon>0$, finite set $\mathcal{F}\subset A$ and apply Corollary 3.3.
%\end{proof}

\begin{thm}[{Homotopy Lifting Theorem; cf.\ \cite[Theorem 5.1]{Bla16}}]\label{thm:HomotopyLifting}
Let $A$ be an $\ell$-open $C^*$-algebra (or more generally, one that satisfies the conclusion of Theorem~\ref{thm:lopen}).
Let $B$ be a $C^*$-algebra, $I$ a closed ideal of $B$, $(\varphi_t)_{t\in [0,1]}$ a point-norm continuous path of $^*$-ho\-mo\-mor\-phisms from $A$ to $B/I$. Suppose $\varphi_0$ lifts to a $^*$-ho\-mo\-mor\-phism $\overline{\varphi_0 } : A \rightarrow B$. Then there is a point-norm continuous path $(\overline{\varphi_t} )_{t\in [0,1]}$ of $^*$-ho\-mo\-mor\-phisms from $A$ to $B$ starting at $\overline{\varphi_0}$ such that $\overline{\varphi_t}$ is a lift of $\varphi_t$ for all $t\in[0,1]$.
\end{thm}
\begin{proof}
	Take an arbitrary finite set $\mathcal F$ of $A$ and real number $\epsilon >0$, and let $\mathcal G,\delta$ be given by Theorem~\ref{thm:lopen}.
We can find a partition $t_0 =0 < t_1 < t_2 < \cdots < t_n =1$ such that $\Vert \varphi_t (a)- \varphi_s (a) \Vert < \delta $ for all $a\in \mathcal{G}$ whenever $t, s \in [t_{i-1} , t_i]$, for any $i$.

Let $D\coloneqq C([0, t_1], B)$ and $J\coloneqq C_0((0,t_1], I)$, which is an ideal of $D$, so that
\begin{equation}\begin{split} 
D/J &\cong C([0,t_1]: B/I) \oplus_{\pi_I} B \\
&= \{(f,b)\in C([0,t_1]: B/I) \oplus B : f(0)=\pi_I (b)\}. \end{split} \end{equation}
Making this identification, define $^*$-ho\-mo\-mor\-phisms $\gamma\coloneqq (\id_{C([0,t_1])}\otimes \varphi_0)\oplus \overline\varphi_0, \theta\coloneqq \varphi|_{[0,t_1]} \oplus \overline\varphi_0: A \rightarrow D/J$ (where $\varphi|_{[0,t_1]}$ denotes the $^*$-ho\-mo\-mor\-phism $A \to C([0,t_1],B/I)$ given by restricting the homotopy $(\varphi_t)$ to $[0,t_1]$).
Then $\gamma$ lifts to the $^*$-ho\-mo\-mor\-phism $\id_{C([0,t_1])} \otimes \overline\varphi$, so by Theorem~\ref{thm:lopen},
$\varphi$ lifts, giving a continuous path of lifts $(\overline{\varphi_t})$ of $(\varphi_t)$ for $t\in [0, t_1]$.
Continuing the same process for successive intervals $[t_1, t_2], \ldots, [t_{n-1}, t_n]$, we get the required continuous path $(\overline{\varphi_t})_{t\in[0,1]}$, such that $\overline\varphi_t$ lifts $\varphi_t$ for all $t\in [0,1]$.
\end{proof}
\begin{prop}\label{prop: GeneralHLT=UnitalHLT}
Let $A$ be an unital $C^*$-algebra. Then $A$ satisfies the conclusion of the Homotopy Lifting Theorem if and only if $A$ satisfies the conclusion in the category of unital $C^*$-algebras and unital $^*$-morphisms.
\end{prop}
\begin{proof}
Suppose $A$ satisfies the conclusion of the Homotopy Lifting Theorem in the category of unital $C^*$-algebras and unital $^*$-morphisms. Let $B$ be a $C^*$-algebra, $I$ a closed ideal of $B$, $(\varphi_t)_{t\in [0,1]}$ a point-norm continuous path of $^*$-ho\-mo\-mor\-phisms from $A$ to $B/I$, and $\overline{\varphi_0 } : A \rightarrow B$ a lift of $\varphi_0$. Set $q_0 \coloneqq \varphi_0 (1), q_1 \coloneqq \varphi_1 (1)$, and $p_0 \coloneqq \overline{\varphi_0}(1)$. Then, $q_0$ is homotopic to $q_1$. Since $\mathbb{C}$ is $\ell$-open, Theorem \ref{thm:HomotopyLifting} implies that there exists a continuous path of projections $(p_t)_{t\in [0,1]}$ connecting $p_0$ and $p_1$ with $q_1 = \pi_I (p_1)$. Consequently, we can find a continuous path of partial isometries $(v_t)_{t\in [0,1]}$ such that 
\begin{equation}
\begin{split}
v_0 &= p_0 , \\
v_t ^* v_t &= p_0 \,\, \forall \,\,t , \\
v_t v_t ^* &= p_t .
\end{split}
\end{equation}
Let $\psi_1 \coloneqq \pi_I (v_1 ^*) \varphi_1 \pi_{I}( v_1) :A \rightarrow q_0 (B/I )q_0$. Then, $(\pi_I (v_t ^*)\varphi_t \pi_I (v_t))_{t\in [0,1]}$ is a point-norm continuous paths of unital $^*$-ho\-mo\-mor\-phisms from $A$ to $q_0 (B/I )q_0$. %connecting $\varphi_0 : A \rightarrow q_0 (B/I )q_0 $ to $\psi$. 
Using the conclusion of the Homotopy Lifting Theorem in the unital category, $\psi_1$ lifts to a unital  $^*$-ho\-mo\-mor\-phism $\overline{\alpha_1}: A \rightarrow p_0 B p_0$ and there is a point-norm continuous path $(\overline{\alpha_t} )_{t\in [0,1]}$ of unital $^*$-ho\-mo\-mor\-phisms connecting $\varphi_0$ to $\overline{\alpha_1}$. Moreover, $\overline{\alpha_t}$ is a lift of $\pi_I (v_t ^*)\varphi_t \pi_I (v_t)$ for each $t\in [0,1]$. Set $\overline{\varphi_t}\coloneqq v_t \alpha_t v_t ^* : A \rightarrow B$. Then, $(\overline{\varphi_t} )_{t\in [0,1]}$ defines a point-norm continuous path of $^*$-ho\-mo\-mor\-phisms from $A$ to $B$ starting at $\overline{\varphi_0}$ such that $\overline{\varphi_t}$ is a lift of $\varphi_t$ for all $t\in[0,1]$. The proof of the converse follows directly from the statement.
\end{proof}
\begin{ex}\label{ex:AFalgebrasHLT} Using Proposition \ref{prop: GeneralHLT=UnitalHLT} and \cite[Theorem 3.5]{Phillips+Raeburn}, unital $AF$-algebras satisfy the condition of the Homotopy Lifting Theorem.
\end{ex}
%Next, we show that every pair $(Z,X)$ with $X$ a topological space and $Z\subset X$ has the homotopy extension property \cite{Hu65} with respect to the space of $^*$-ho\-mo\-mor\-phism from a separable $\ell$-open $C^*$-algebra. 
%\begin{cor}
%Let $A$ be a separable $\ell$-open $C^*$-algebra, $B$ a $C^*$-algebra, $X$ a topological space, $Z$ a subspace of $X$, and $\beta_t : Z \rightarrow Hom (A, B) (0\leq t\leq 1) $ a homotopy from $\beta_0$ to $\beta_1$. Suppose $\beta_0 $ extends to a continuous function $\overline{\beta_0}$ from $X$ to $Hom (A, B)$. Then $\beta_1$ extends to a continuous function $\overline{\beta_1} :X \rightarrow Hom(A, B)$ and there exists a homotopy $(\overline{\beta_t}) (0 \leq t \leq 1)$ that extends $\beta_t$.
%\end{cor}
%\begin{proof}
%The path $(\beta_t ) (0 \leq t \leq 1)$ from $\beta_0$ to $\beta_1$ defines a point-norm continuous path of $^*$-ho\-mo\-mor\-phism $\varphi_{t} : A \rightarrow C_{0}(Z, B)$. Moreover, $\varphi_0$ lifts to a $^*$-ho\-mo\-mor\-phism $\overline{\varphi}_0 : A \rightarrow C_{0}(X, B)$ since $\beta_0 $ extends to $\overline{\beta_0}$. By the homotopy lifting theorem (Theorem 3.4), $\varphi_{t}$ lifts to a point-norm continuous path of $^*$-ho\-mo\-mor\-phism starting at $\varphi_0$. Then, the result follows.
%\end{proof}
%\begin{cor}
%		Let $A$ be an $\ell$-open $C^*$-algebra, $B$ a $C^*$-algebra, $I$ a closed ideal of $B$, and $\varphi$ a $^*$-ho\-mo\-mor\-phism from $A$ to $B/I$. If $\varphi$ is homotopic to a $^*$-ho\-mo\-mor\-phism which lifts to B, then $\varphi$ lifts to $B$.
%\end{cor}
\begin{rem} Conway (\cite{Conway75, Conway77}) studied a restricted version of the homotopy lifting theorem, which he called the $C^*$-covering homotopy property. He considered Theorem \ref{thm:HomotopyLifting} in the case where $B/I$ is the Calkin algebra. 

%Phillips and Raeburn \cite{Phillips+Raeburn} proved that the homotopy lifting theorem holds for $M_{2^\infty}$ (and generally $AF$-algebras) in the category of unital $C^*$-algebras and unital $^*$-morphisms. It is still unclear whether the conclusion of the homotopy lifting theorems holds for $M_{2^\infty}$ in the general category. \AT{I think this can be solved. Rewrite or delete this sentence.}
\end{rem}
Combining all the previous theorems and corollaries, we have the following characterization of $\ell$-open $C^*$-algebra.
\begin{thm}\label{thm:Main}
		Let $A$ be a $C^*$-algebra. Then the following are equivalent
		\begin{enumerate}
			\item $A$ is $\ell$-open.
			\item The system of maps $\Hom(A,B) \to \Hom(A,B/J)$ (over all $C^*$-algebras $B$ and ideals $J$) is uniformly open, as in the conclusion of Theorem~\ref{thm:lopen}
			\item $A$ satisfies the conclusion of the Homotopy Lifting Theorem (Theorem~\ref{thm:HomotopyLifting}) and $\Hom (A, B)$ is locally path-connected for all $C^*$-algebras $B$.
		\end{enumerate}
\end{thm}
\begin{proof}
(i)$\Rightarrow$(ii) is Theorem~\ref{thm:lopen} and (ii)$\Rightarrow$(iii) is by Corollary~\ref{cor:PathConnected} and Theorem~\ref{thm:HomotopyLifting}.

To prove that (iii)$\Rightarrow$(i), let $\phi_n : A \rightarrow B/I$ be a net of $^*$-ho\-mo\-mor\-phisms which converges point-norm to a liftable $^*$-ho\-mo\-mor\-phism $\phi : A \rightarrow B/I$. Since $\Hom(A, B/I)$ is locally path-connected, $\phi_n$ is homotopic to $\phi$ for sufficiently large $n$.
The conclusion of the Homotopy Lifting Theorem then implies that $\phi_n$ is liftable for these $n$.
This shows that $\Hom(A,B,I)$ is open in $\Hom(A,B/I)$, as required.
\end{proof}

\begin{ex} 
Satisfying the condition of the Homotopy Lifting Theorem doesn't guarantee $\ell$-openness of $C^*$-algebras. $M_{2^\infty}$ satisfies the condition of the Homotopy Lifting Therem (see Example \ref{ex:AFalgebrasHLT}), but it is not an $\ell$-open $C^*$-algebra. To see that $M_{2^\infty}$ is not $\ell$-open, suppose otherwise. 
Using any finite set $\mathcal F \subseteq M_{2^\infty}$ and any $\epsilon>0$, obtain $\delta >0$ and a finite set $\mathcal G\subset M_{2^ \infty}$ according to Theorem \ref{thm:lopen}. Without loss of generality, we can assume $\mathcal G \subset M_{2^k}$ for some $k$.

Let us set $B\coloneqq B(\mathcal H)$ and $J\coloneqq \mathcal K$, so that $B/J$ is the Calkin algebra.
Let $\phi_1,\phi_2:A \to B/J$ be $^*$-ho\-mo\-mor\-phisms such that $\phi_1$ is liftable but $\phi_2$ is not (which exists by \cite{Thayer75}).
Define $\varphi_i \coloneqq \id_{M_{2^k}} \otimes \phi_i : M_{2^k} \otimes M_{2^\infty}\cong M_{2^\infty} \rightarrow M_{2^k} \otimes (B/J) \cong (M_{2^k} \otimes B)/(M_{2^k} \otimes J)$.
Then we have that $\varphi_1(a) = \varphi_2 (a)$ for all $a \in \mathcal G$.
Hence, Theorem~\ref{thm:lopen} tells us that since $\varphi_1$ is liftable, so is $\varphi_2$. 
The Ext-class of $\varphi_2$ is $2^k$ times the Ext-class of $\phi_2$; $Ext(M_{2^\infty})$ is the $2$-adic integers, which is torsion-free, it follows that $\varphi_2$ is not liftable, a contradiction.
Hence, $M_{2^\infty}$ is not $\ell$-open.
\end{ex}

The characterization of $\ell$-openness confirms a conjecture of Blackadar \cite[Page 299]{Bla16}, as follows. 

\begin{cor}\label{cor:openimpliesclosed}
	Let $A$ be an $\ell$-open $C^*$-algebra. Then $A$ is $\ell$-closed.
\end{cor}

\begin{proof}
Fix a $\epsilon>0$ and a finite set $\mathcal{F}$ and choose a $\delta>0$ and finite set $\mathcal{G}$ as in Theorem~\ref{thm:lopen}.
Let $\phi_n : A \rightarrow B/I$ be a net of liftable $^*$-ho\-mo\-mor\-phisms which converges point-norm to a $^*$-ho\-mo\-mor\-phism $\phi : A \rightarrow B/I$. We can find $m$ such that $\Vert \phi_m (u) - \phi (u) \Vert < \delta$ for all $u\in \mathcal{G}$. Since $\phi_m$ is liftable, the conclusion of Theorem \ref{thm:lopen} implies that $\phi$ is liftable. Hence, $A$ is $\ell$-closed.
\end{proof}

\section{Characterization of $\ell$-closed $C^*$-algebras}

We now characterize $\ell$-closed $C^*$-algebras, showing that the condition is equivalent to a uniform relative openness of the map $\Hom(A,B) \to \Hom(A,B/I)$.
We require separability for this characterization, and one direction uses a Cauchy sequence argument.

\begin{thm}\label{thm:MainClosed}
%Let $A$ be a $C^*$-algebras generated by a finite or countable set $\mathcal{G} = \{a_1, a_2 , \ldots \}$ with $\lim\limits_{n \rightarrow \infty } \Vert a_n \Vert =0$ if $\mathcal{G}$ is infinite. 
Let $A$ be a separable $C^*$-algebra. Then the following are equivalent:
\begin{enumerate}
\item $A$ is $\ell$-closed.
\item For any $\epsilon > 0$ and finite set $\mathcal{F} \subset A$, there is a $\delta >0$ and a finite set $\mathcal{G} \subset A$ such that whenever $B$ is a $C^*$-algebra, $I$ is a closed ideal of $B$, $\psi$ and $\phi$ are $^*$-ho\-mo\-mor\-phisms from $A$ to $B$ with $\Vert \pi_{I} \circ \phi (u) - \pi_{I} \circ \psi(u) \Vert < \delta$ for all $u\in \mathcal{G}$, then there exists a $^*$-ho\-mo\-mor\-phism $\eta: A \rightarrow B$ such that $\Vert \phi(v) - \eta(v) \Vert < \epsilon$ for all $v\in \mathcal{F}$ and $\pi_{I} \circ \psi =\pi_{I} \circ \eta$. 
\end{enumerate}
\end{thm}
\begin{proof}
(i)$\Rightarrow$(ii).
Let $(\mathcal{G}_n )$ be an increasing sequence of finite subsets of $A$ whose union is dense in $A$. Suppose (ii) is false for a fixed $\epsilon >0$ and finite set $\mathcal{F}\subset A$. Then, there are $C^*$-algebras $B_n$ with ideals $I_n$, and $^*$-ho\-mo\-mor\-phisms $\phi_n, \psi_n : A \rightarrow B_n $ such that 
\begin{equation} \label{eq:phipsiclose} \Vert \pi_{I_n} \circ \phi_n (a) - \pi_{I_n} \circ \psi_n(a) \Vert < \frac{1}{n}\quad\text{for all }a\in \mathcal{G}_n, \end{equation}
 but no $^*$-ho\-mo\-mor\-phism $\eta_n :A \rightarrow B_n$ satisfies both $\Vert \phi_n (a) - \eta_n(a) \Vert < \epsilon$ for all $a\in \mathcal{F}$ and $\pi_{I_n} \circ \psi_n = \pi_{I_n} \circ \eta_n $.
	
	Let $B \coloneqq \prod \limits_{n=1}^{\infty} B_n$, $I \coloneqq \prod \limits_{n=1}^{\infty} I_n$, and $J\coloneqq \bigoplus \limits_{n=1}^{\infty} B_n$. Define $^*$-ho\-mo\-mor\-phisms $\overline{\phi}\coloneqq(\phi_1,\phi_2,\dots),\overline\psi\coloneqq(\psi_1,\psi_2,\dots): A \rightarrow B$.

	By \eqref{eq:phipsiclose}, it follows that $\pi_{I+J}\circ\overline\phi = \pi_{I+J}\circ\overline\psi$.
Then by the general Chinese remainder theorem (Lemma \ref{lem:CRT}), there exists a $^*$-ho\-mo\-mor\-phism $\theta : A \rightarrow B/ (I\cap J)$ such that 
	\begin{equation}
\label{eq:thetadef2}
		\pi_{J} \circ \overline{\phi} = \pi_{J} \circ \theta\,\, \text{ and}\,\, \pi_{I} \circ \overline{\psi} = \pi_{I} \circ \theta 
	\end{equation}
%Take a $^*$-linear lift $(\theta_n)_{n=1}^\infty:A \to B$ of $\theta$ (which need not be a $^*$-ho\-mo\-mor\-phism), thus defining $\theta_n:A \to B_n$.
For each $n\in\mathbb N$, define the $^*$-ho\-mo\-mor\-phism
	\begin{equation}
		\overline\alpha_n\coloneqq  (\psi_1, \psi_2, \ldots, \psi_{n-1}, \phi_{n}, \phi_{n+1}, \ldots):A \to B.
	\end{equation}
Then by the definition of $J$, we have $\pi_J \circ \overline\alpha_n = \pi_J \circ \overline\phi$.
Therefore by \eqref{eq:thetadef2}, for $x \in A$, 
%On the other hand, for $x \in A$,
\begin{equation} \begin{split}
\|\pi_{I\cap J}\circ\overline\alpha_n(x)-\theta(x)\| &= 
\|\pi_I \circ \overline\alpha_n(x) - \pi_I \circ \overline\psi(x)\| \\
&= \sup_{m\geq n} \|\pi_{I_m}\circ\phi_m(x)-\pi_{I_m}\circ\psi_m(x)\| \to 0. 
\end{split} \end{equation}
Since $A$ is $\ell$-closed, we deduce that $\theta$ lifts to a $^*$-ho\-mo\-mor\-phism $\eta=(\eta_1,\eta_2,\dots): A \rightarrow B$.
Then \eqref{eq:thetadef2} implies that $\pi_{I_n} \circ \psi_n = \pi_{I_n} \circ \eta_n $ and $\lim \limits_{n \rightarrow \infty }\Vert \phi_n (x)- \eta_n (x) \Vert =0$ for all $x\in A$. 
%Since $\lim \limits_{j \rightarrow \infty}\Vert  a_j \Vert =0$, there exists $m$ such that $\Vert a_j \Vert < \frac{\epsilon}{2}$ for all $j> m$. 
Hence, there is a $k$ such that
%For each $j= 1, 2,\ldots, m$, we can find $n_j$ with 
\begin{equation}\Vert \phi_k (a) -\eta_{k}(a) \Vert < \epsilon \end{equation}
for all $a\in \mathcal{F}$. This is a contradiction.
%For $j> m$, $\Vert \phi_k (a_j) -\eta_{k}(a_j) \Vert < \Vert \phi_(k)(a_j)\Vert + \Vert \eta_k (a_j) \Vert < \frac{\epsilon}{2}+ \frac{\epsilon}{2} = \epsilon$. Therefore, $\Vert \phi_k (a_j) -\eta_{k}(a_j) \Vert < \epsilon $ for all $j$ and a contradiction.

(ii)$\Rightarrow$(i). Suppose $\eta_n : A\rightarrow B/I$ is a sequence of liftable $^*$-ho\-mo\-mor\-phisms which converges pointwise to a $^*$-ho\-mo\-mor\-phism $\eta : A \rightarrow B/I$. Let $\mathcal{F}_n$ be an increasing sequence of finite sets whose union is dense in $A$. Choose $\delta_n >0$ and a finite set $\mathcal{G}_n$ such that they satisfy the conditions of (ii) with $\epsilon\coloneqq \frac{1}{2^n}$ and $\mathcal F\coloneqq \mathcal{F}_n$.
By passing to a subsequence, we may assume without loss of generality that 
\begin{equation}\Vert \eta_n (u)-\eta_{n+1} (u)\Vert < \delta_n \,\,\,\,\, \text{for all}\,\, u\in \mathcal{G}_n.\end{equation}
Let $\overline{\eta}_n : A\rightarrow B$ be a lift of $\eta_n$.
Then by the choice of $\mathcal G_1$ and $\delta_1$ from (ii) implies that there exists a $^*$-ho\-mo\-mor\-phism $\xi_2 : A\rightarrow B$ such that $\Vert \overline{\eta}_1 (v) - \xi_2(v) \Vert < \frac{1}{2}$ for all $v\in \mathcal{G}_1$ and $\pi_I \circ \overline{\eta}_2 = \pi_I \circ \xi_2$. 
Then we have $ \Vert \pi_I \circ \overline{\eta}_2(u) - \pi_I \circ \overline{\eta}_3 (u) \Vert =  \Vert \pi_1 \circ \xi_2 (u) - \pi_I \circ \overline{\eta}_3(u) \Vert < \delta_2 $ for all $u\in \mathcal{G}_2$. 
Using the choice of $\mathcal G_2$ and $\delta_2$ from (ii), we have a $^*$-ho\-mo\-mor\-phism $\xi_3 : A\rightarrow B$ such that $\Vert \xi_2(v) - \xi_3(v)  \Vert < \frac{1}{2^2}$ and $\pi_I \circ \overline{\eta}_3 = \pi_I \circ \xi_3$. Continuing the process and setting $\xi_1 = \overline{\eta}_1 $, we get a $(\xi_n :A \rightarrow B)$ such that $\Vert \xi_n (a) - \xi_{n+1}(a) \Vert < \frac{1}{2^n}$ for all $a\in F_n$ and $\eta_n = \pi_{I} \circ \xi_n$. Consequently, the sequence $(\xi_n (a))_{n=1}^\infty$ is Cauchy for each $a \in A$, so it converges to some $\xi(a) \in B$.
This defines a $^*$-ho\-mo\-mor\-phism $\xi: A\rightarrow B$, and for $a\in A$,
\begin{equation} \begin{split}
\pi_I \circ \xi(a) = \lim_n \pi_I \circ \xi_n(a) = \lim_n \eta_n(a) = \eta(a).
\end{split} \end{equation}
Therefore we obtain a lift of $\eta$, and this shows that $A$ is $\ell$-closed.
\end{proof}

Note that condition (iii) of Theorem~\ref{thm:Main} strengthens condition (ii) in Theorem~\ref{thm:MainClosed}, by replacing $\psi:A \to B$ with a map $A \to B/I$ which is (a priori) not liftable.
This gives a quick proof of Corollary~\ref{cor:openimpliesclosed} in the separable case.

Theorem \ref{thm:MainClosed} may be reformulated as follows.
\begin{thm}\label{thm:MainClosedReformulation}Let $A$ be a separable $C^*$-algebra and $S$ a generating set of $A$. Then the following are equivalent:
\begin{enumerate}
\item $A$ is $\ell$-closed.
\item For any $\epsilon > 0$ and finite set $\mathcal{F} \subset S$, there is a $\delta >0$ and a finite set $\mathcal{G} \subset S$ such that whenever $B$ is a $C^*$-algebra, $I$ is a closed ideal of $B$, $\psi$ and $\phi$ are $^*$-ho\-mo\-mor\-phisms from $A$ to $B$ with $\Vert \pi_{I} \circ \phi (u) - \pi_{I} \circ \psi(u) \Vert < \delta$ for all $u\in \mathcal{G}$, then there exists a $^*$-ho\-mo\-mor\-phism $\eta: A \rightarrow B$ such that $\Vert \phi(v) - \eta(v) \Vert < \epsilon$ for all $v\in \mathcal{F}$ and $\pi_{I} \circ \psi =\pi_{I} \circ \eta$. 
\end{enumerate}
\end{thm}
In \cite[Example 6.4]{Bla16}, Blackadar asks whether $C^*(\mathbb{F}_\infty)$, the universal $C^*$-algebra generated by a sequence of unitaries, is $\ell$-closed. We now show that it is.
\begin{ex} $C^*(\mathbb{F}_\infty)$ is $\ell$-closed. To see this, consider $\epsilon >0$, a finite set $\mathcal{F} \subset \{u_1, u_2, \ldots \}$, an ideal $I$ of $B$, and $^*$-ho\-mo\-mor\-phisms $\phi, \psi : C^*(\mathbb{F}_\infty) \rightarrow B$. Without loss of generality, we may assume $F=\{u_1, u_2, \ldots, u_n\}$ for some $n$. $C^*(F) \cong C^* (\mathbb{F}_n)$ is semiprojective ((this is well-known; see \cite[Corollary 2.22 and Proposition 2.31]{Bla85} for example) and so $\ell$-closed by \cite[Corollary 6.2]{Bla16}. Choose $\delta >0$ and a finite set $\mathcal{G}\subset \mathcal{F}$ as in Theorem \ref{thm:MainClosedReformulation} (applied to $C^*(F)$). Then $\Vert \pi_{I} \circ \phi (u)-\pi_{I} \circ \psi (u)\Vert < \delta$ for all $u\in \mathcal{G}$ implies there exists a $^*$-ho\-mo\-mor\-phism $\xi : C^* (\mathbb{F}_n) \rightarrow B$ such that $\Vert \phi(v) - \xi (v) \Vert < \epsilon$ for all $v\in \mathcal{F}$ and $\pi_{I}\circ \xi = \pi_{I} \circ \psi \vert_{C^*(\mathbb{F}_n)}$. $\eta : C^*(\mathbb{F}_\infty) \rightarrow B$ defined by
\begin{equation}
\eta(u_m) \coloneqq  \begin{cases}
\xi(u_m), \quad & m\leq n ;\\
\psi (u_m), \quad & m>n
\end{cases}
\end{equation}
is a $^*$-ho\-mo\-mor\-phism satisfying  $\Vert \phi(v) - \eta (v) \Vert < \epsilon$ for all $v\in \mathcal{F}$ and  $\pi_{I}\circ \xi = \pi_{I} \circ \eta$, as required.
\end{ex}
\begin{rem}
Using a result of Salinas, we can see that $\ell$-closedness in often very restrictive, just by looking at the case $B\coloneqq \mathcal B(\ell^2)$ and $I\coloneqq \mathcal K$.
Here, $\Hom(A,B/I)$ corresponds to extensions of $A$ by the compact operators, with $\Hom(A,B,I)$ corresponding to the subset of trivial extensions.
Therefore, if $A$ is $\ell$-closed then the set of trivial extensions is closed in the set of all extensions (using the corresponding topology).

When $A$ is a quasidiagonal nuclear $C^*$-algebra, then the closure of the set of trivial extensions is the set of all quasidiagonal extensions by \cite[Theorem 2.9]{Salinas} (the topology is described in \cite[Remark 2.8]{Salinas}, and can be seen to agree with the point-topology in $\Hom(A,B/I)$).
Therefore if $A$ is additionally $\ell$-closed, then all quasidiagonal extensions must be trivial.

From this we may conclude that any infinite dimensional UHF algebra is not $\ell$-closed, since for these $C^*$-algebras, all extensions are quasidiagonal, but they are not all trivial (see \cite[Remark 2.13]{Salinas}).
\end{rem}
\section{Commutative unital $\ell$-open $C^*$-algebras}
In this section, we show that commutative unital separable $\ell$-open $C^*$-algebras coincide with commutative unital separable semiprojective $C^*$-algebras.
We begin with the following which may be of independent interest.

%\begin{thm}[\cite{}]
%Let $Y$ be a compact metric space. The following are equivalent:
%\begin{enumerate}
%	\item $Y$ is an ANR
%	\item $Y$ is a $e$-open space
%	\item Let $p\in X$ and $U$ a neighborhood of $p$. Then, there exists a neigbhorhood $V \subset U$ of $p$ such that every continuous mapping $g: Z \rightarrow V$ of a closed subset $Z$ of a given metric space $X$ into $V$ can be extended to a continuous mapping $\overline{g}: X \rightarrow U$.
%\end{enumerate}
%\end{thm}
\begin{prop}
\label{prop:lopenWeaklysemiproj}
Let $A$ be an $\ell$-open $C^*$-algebra and $\psi : A\rightarrow B$ a weakly semiprojective $^*$-ho\-mo\-mor\-phism. Then $\psi$ is a semiprojective $^*$-ho\-mo\-mor\-phism.
\end{prop}
\begin{proof}
%Assume that $A$ is generated by $\{a_1 , a_2, \ldots\}$ with $\lim \limits_{j \rightarrow \infty} \Vert a_j \Vert =0$. 
Fix $\epsilon>0$ and a finite set $\mathcal{F}$ in $A$, and let $\delta>0$ and $\mathcal{G}\subset A$ be given by Theorem \ref{thm:lopen}.
%Then, there exists $m\in \mathbb{N}$ such that $\Vert a_j \Vert < \frac{\delta}{2} $ for all $j> m$. Let $F= \{a_1, a_2 , \ldots , a_m\}$. 
Given any $^*$-ho\-mo\-mor\-phism  $\varphi : B \rightarrow C/\overline{\bigcup{}_{n} \ J_{n}}$ with $J_1 \triangleleft J_2 \triangleleft \cdots \triangleleft C$ an increasing sequence of closed ideals of a $C^*$-algebra $C$, by weak semiprojectivity we can find some $n$ and a $^*$-ho\-mo\-mor\-phism $\phi : A \rightarrow C/J_n$ such that  
\begin{equation} \Vert \varphi \circ \psi (u) - \pi \circ \phi (u) \Vert < \delta \end{equation}
for all $u\in \mathcal{G}$. %Since $\vert a_j \Vert < \frac{\delta}{2}$ for all $j>m$, we have that $$\Vert \varphi \circ \psi (a_j) - \pi \circ \phi (a_j) \Vert < \delta$$
%for all $j$. 
It follows from Theorem \ref{thm:lopen} that there exists a $^*$-ho\-mo\-mor\-phism $\rho : A\rightarrow C/J_n$ such that $\varphi \circ \psi = \pi \circ \rho$. Hence $ \psi$ is a semiprojective $^*$-ho\-mo\-mor\-phism.
\end{proof}
%\begin{defn}
%	\begin{enumerate}
%		\item Let $\mathcal{U}$ be a collection in a topological space $X$ and $p$ a point of $X$. Then, the order of $\mathcal{U}$ at $p$ is the number of members of $\mathcal{U}$ which contain $p$, and it is denoted by $Ord_p (\mathcal{U})$. If there exists infintely many such members, then $Ord_p (\mathcal{U}) = +\infty$.
%		$$\text{ The Order of}\,\, \mathcal{U}, \,Ord (\mathcal{U}), = \sup \{Ord_p (\mathcal{U}): p\in X\}.$$
%	\item If every finite open cover $\mathcal{U}$ of a topological space $X$ can be refined by a finite open cover $\mathcal{V}$ such that $Ord (\mathcal{V}) \leq n+1$, then $X$ has covering dimension $\leq n$, $dim (X) \leq n$.
%$X$ has a local covering dimension $locdim(X) \leq n$ if every point $x\in X$ has a closed neigbhorhood $D$ with covering dimension $dim(D) \leq n$ (see \cite{Nagami} for more details).
%	\end{enumerate}
%\end{defn}
%Note that if $X$ is paracompact (e.g compact, locally compact, $\sigma$-compact), then $locdim(X) = dim (X)$.
\begin{lem}[\cite{CD10}, Proposition 3.1]\label{lem:S1copy}
	Let $X$ be a compact, connected, and locally connected metric space, of covering dimension $>1$. Then $X$ contains a topological copy of the circle $S^1$.	
\end{lem}
Recall that metrizable space $X$ is an absolute neighbourhood retract (ANR) if, for any metrizable space $Y$ and closed subspace $Z$ of $Y$, there exists a neighbourhood $V$ of $Z$ such that any continuous map $\eta: Z \to X$ extends to a continuous map $\theta: V \to X$.
\begin{prop}Let $X$ be a compact ANR. If $C(X)$ satisfies the condition of the homotopy lifting theorem, then $dim(X)\le 1$.
\end{prop}
\begin{proof}
We prove that the result along the lines of S{\o}rensen and Thiel's proof of \cite[Proposition 3.1]{ST12}. Suppose by contradiction that $\dim(X)\ge 2$. Since $X$ is compact, we have that $\mathrm{locdim} (X)=\dim (X) \ge 2$, which implies that there is an $x_0 \in X$ such that $\dim (D) \ge 2$ for every closed neighbourhood $D$ of $x_0$ (see \cite{Nagami} for details on $\mathrm{locdim}(X)$). Let $D_1,D_2,\dots$ be a decreasing sequence of closed neighbourhoods of $x_0$ with $\dim(D_k)\ge 2$ for all $k$. Using Lemma \ref{lem:S1copy}, there exists a topological embedding $\psi_k : S^1 \xhookrightarrow{} D_k  \subset X$ for each $k$.
Let 
\begin{equation}Y\coloneqq (0,0) \cup \bigcup _{k\ge 1} S((\frac{1}{2^k},0), \frac{1}{4\cdot 2^k}) \subset \mathbb{R}^2,\end{equation} 
where $S(x, r)$ is the circle centered at $x$ of radius $r$. Then $C(Y)$ is weakly semiprojective (\cite{ST12}]). Define $\psi : Y \rightarrow X$ to send $(0,0)$ to $x_0$ and to be $\psi_k$ on the circle $S((\frac{1}{2^k},0), \frac{1}{4. 2^k})$.
Then $\psi$ induces a $^*$-ho\-mo\-mor\-phism $\psi^* : C(X)\rightarrow C(Y)$, which is weakly semiprojective since $C(Y)$ is.

Let $ \mathcal{T}$ be the Toeplitz algebra and let $\mathcal{K}$ be the ideal of compact operators.
Writing $A^\dag$ for the unitization of $A$, set
\begin{equation}\begin{split}
B & \coloneqq (\bigoplus_{k\ge 1} \mathcal{T})^\dag \\
&= \{ (t_1, t_2, \ldots ,) \in \prod_{k\ge 1}\mathcal{T}:(t_k)_k \,\, \text{converges to a scalar multiple of} \, 1_\mathcal{T}\}
\end{split}\end{equation}
and $J_k \coloneqq \underbrace{\mathcal{K}\oplus \mathcal{K}\oplus \cdots \mathcal{K}}_{k \,\,times} \oplus 0 \oplus 0 \cdots $. Then $J_k \subset J_{k+1}$, $J= \overline{\bigcup_k J_k} = \bigoplus_{k\ge 1} \mathcal{K}
$ ,
\begin{equation} B/J_k = \underbrace{C(S^1)\oplus C(S^1)\oplus \cdots \oplus C(S^1)}_{k \,\,times} 	\oplus (\bigoplus_{l \ge k+1} \mathcal{T})^\dag,\end{equation}
and $B/J =(\bigoplus_{k\ge 1}(C(S^1)))^\dag \cong C(Y)$.
	%		\item $B/J_k = \underbrace{C(S^1)\oplus C(S^1)\oplus \cdots \oplus C(S^1)}_{k \,\,times} 	\oplus (\oplus_{l \ge (k+1)} \mathcal{T})^\dag$ and  $B/J = (\oplus_{l \ge 1} C(S^1))^\dag \cong C(Y)$.

%Proposition \ref{prop:lopenWeaklysemiproj} implies $\psi ^*$ is a semiprojective $^*$-ho\-mo\-mor\-phism, so $\psi ^*$ lifts to some $\overline{\psi} : C(X) \rightarrow B/J_k$.
	
	%= \{ (b_1 , b_2, \ldots )\in \prod \mathcal{T}: \text{such that}\, (b_k) \,\text{converges to a scalar multiple of} 1_{\mathcal{T}} \}$
	%	\end{itemize}
%\end{frame}
%\begin{frame}
%\begin{itemize}

Since $X$ is a compact ANR, we can find $\epsilon >0$ and finite set $\mathcal{F}$ in $C(X)$ such that any unital $^*$-ho\-mo\-mor\-phisms $\varphi_1 , \varphi_2 : C(X) \rightarrow C(Y)$ with $\Vert \varphi_1 (f)-\varphi_2 (f) \Vert < \epsilon$ for all $f\in \mathcal{F}$ are homotopic \cite[Theorem IV.1.1]{Hu65}. The weakly semiprojectivity of $\psi^*$ implies there exists a $^*$-ho\-mo\-mor\-phism $\theta : C(X) \rightarrow B/J_k$ such that $\Vert \psi^* (f) -\pi  \theta(f) \Vert < \epsilon$ for all $f\in \mathcal{F}$. It follows that $\psi^*$ is homotopic to $\pi\theta$. We apply the homotopy lifting theorem to say $\psi^*$ lifts to some $\overline{\psi}: C(X) \rightarrow B/J_k$. 
\tikzset{every picture/.style={line width=0.75pt}} %set default line width to 0.75pt        

\begin{tikzpicture}[x=0.75pt,y=0.75pt,yscale=-1,xscale=1]

%Straight Lines [id:da13767236822102924] 
\draw    (127,178) -- (149,178) -- (186,178) ;
\draw [shift={(188,178)}, rotate = 180] [color={rgb, 255:red, 0; green, 0; blue, 0 }  ][line width=0.75]    (10.93,-3.29) .. controls (6.95,-1.4) and (3.31,-0.3) .. (0,0) .. controls (3.31,0.3) and (6.95,1.4) .. (10.93,3.29)   ;
%Straight Lines [id:da27509119613671595] 
\draw    (228,178) -- (293,178.97) ;
\draw [shift={(295,179)}, rotate = 180.86] [color={rgb, 255:red, 0; green, 0; blue, 0 }  ][line width=0.75]    (10.93,-3.29) .. controls (6.95,-1.4) and (3.31,-0.3) .. (0,0) .. controls (3.31,0.3) and (6.95,1.4) .. (10.93,3.29)   ;
%Straight Lines [id:da9166038307898368] 
\draw    (308,120) -- (308.96,167) ;
\draw [shift={(309,169)}, rotate = 268.83] [color={rgb, 255:red, 0; green, 0; blue, 0 }  ][line width=0.75]    (10.93,-3.29) .. controls (6.95,-1.4) and (3.31,-0.3) .. (0,0) .. controls (3.31,0.3) and (6.95,1.4) .. (10.93,3.29)   ;
%Straight Lines [id:da31605799056967077] 
\draw    (332,180) -- (404,181.95) ;
\draw [shift={(406,182)}, rotate = 181.55] [color={rgb, 255:red, 0; green, 0; blue, 0 }  ][line width=0.75]    (10.93,-3.29) .. controls (6.95,-1.4) and (3.31,-0.3) .. (0,0) .. controls (3.31,0.3) and (6.95,1.4) .. (10.93,3.29)   ;
%Straight Lines [id:da2698158412930798] 
\draw    (338,109) -- (410,110.95) ;
\draw [shift={(412,111)}, rotate = 181.55] [color={rgb, 255:red, 0; green, 0; blue, 0 }  ][line width=0.75]    (10.93,-3.29) .. controls (6.95,-1.4) and (3.31,-0.3) .. (0,0) .. controls (3.31,0.3) and (6.95,1.4) .. (10.93,3.29)   ;
%Straight Lines [id:da5312505236629717] 
\draw    (428.11,119.02) -- (428,168) ;
\draw [shift={(428,170)}, rotate = 270.12] [color={rgb, 255:red, 0; green, 0; blue, 0 }  ][line width=0.75]    (10.93,-3.29) .. controls (6.95,-1.4) and (3.31,-0.3) .. (0,0) .. controls (3.31,0.3) and (6.95,1.4) .. (10.93,3.29)   ;
%Curve Lines [id:da5531406445963352] 
\draw    (104,191) .. controls (129.87,220.85) and (318.21,257.65) .. (420.57,197.93) ;
\draw [shift={(422.11,197.02)}, rotate = 149.12] [color={rgb, 255:red, 0; green, 0; blue, 0 }  ][line width=0.75]    (10.93,-3.29) .. controls (6.95,-1.4) and (3.31,-0.3) .. (0,0) .. controls (3.31,0.3) and (6.95,1.4) .. (10.93,3.29)   ;
%Straight Lines [id:da5169452163014476] 
\draw  [dash pattern={on 4.5pt off 4.5pt}]  (117,163.6) -- (285.08,113.17) ;
\draw [shift={(287,112.6)}, rotate = 163.3] [color={rgb, 255:red, 0; green, 0; blue, 0 }  ][line width=0.75]    (10.93,-3.29) .. controls (6.95,-1.4) and (3.31,-0.3) .. (0,0) .. controls (3.31,0.3) and (6.95,1.4) .. (10.93,3.29)   ;

% Text Node
\draw (87,167.4) node [anchor=north west][inner sep=0.75pt]    {$C( X)$};
% Text Node
\draw (190,169.4) node [anchor=north west][inner sep=0.75pt]    {$C( Y)$};
% Text Node
\draw (297,170.4) node [anchor=north west][inner sep=0.75pt]    {$B/J$};
% Text Node
\draw (421,101.4) node [anchor=north west][inner sep=0.75pt]    {$\mathcal{T}$};
% Text Node
\draw (293,99.4) node [anchor=north west][inner sep=0.75pt]    {$B/J_{k}$};
% Text Node
\draw (406,171.4) node [anchor=north west][inner sep=0.75pt]    {$C\left( S^{1}\right)$};
% Text Node
\draw (244,165.4) node [anchor=north west][inner sep=0.75pt]    {$\cong $};
% Text Node
\draw (145,180.4) node [anchor=north west][inner sep=0.75pt]    {$\psi {^{*}}{}$};
% Text Node
\draw (343,181.4) node [anchor=north west][inner sep=0.75pt]    {$\rho _{k}{}_{+}{}_{1}$};
% Text Node
\draw (352,97.4) node [anchor=north west][inner sep=0.75pt]    {$\sigma _{k}{}_{+}{}_{1}$};
% Text Node
\draw (268,231.4) node [anchor=north west][inner sep=0.75pt]    {$\psi _{k+1}^{*}$};
% Text Node
\draw (180,120.4) node [anchor=north west][inner sep=0.75pt]    {$\overline{\psi }$};

\end{tikzpicture}

Let $\sigma_{k+1} : B/J_k \rightarrow \mathcal{T}$ be the projection of $B/J_k$ onto the (k+1)-th coordinate and $\rho_{k+1}: B/J \rightarrow C(S^1)$ be the projection of $B/J$ onto the (k+1)-th coordinate. Note that $\rho_{k+1} \circ \psi^* : C(X)\rightarrow C(S^1)$ coincide with the $^*$-ho\-mo\-mor\-phism induced by $\psi_{k+1}: S^1 \xhookrightarrow{} D_{k+1} \subset X$ and it is surjective since $\psi_{k+1}$ is an inclusion.
The generating unitary of $C(S^1)$ lifts to a normal element in $C(X)$ under $\psi^{*}_{k+1}$, but it does not lift to a normal element in $\mathcal{T}$, which is a contradiction. Hence, $\dim(X) \leq 1$.

%Let $Y$, $B$, $J$, $J_k$, and $\psi^*$ be as in Theorem 5.3. 
%Since $X$ is a compact ANR, we can find $\epsilon >0$ and finite set $\mathcal{F} \subset C(X)$ such that any unital $^*$-ho\-mo\-mor\-phisms $\varphi_1 , \varphi_2 : C(X) \rightarrow C(Y)$ with $\Vert \varphi_1 (f)-\varphi_2 (f) \Vert < \epsilon$ for all $f\in \mathcal{F}$ are homotopic \cite[Theorem 1.1]{Hu65}. The weakly semiprojectivity of $\psi^*$ implies there exists a $^*$-ho\-mo\-mor\-phism $\theta : C(X) \rightarrow B/J_k$ such that $\Vert \psi^* (f) -\pi_k  \theta(f) \Vert < \epsilon$ for all $f\in \mathcal{F}$. It follows that $\psi^*$ is homotopic to $\pi_k\theta$. We apply the homotopy lifting theorem to say $\psi^*$ lifts to some $\overline{\psi}: C(X) \rightarrow B/J_k$. This leads to a contradiction as in Theorem 5.3.
\end{proof}
\begin{thm}
Let $X$ be a compact metric space. Then the following are equivalent
\begin{enumerate}
	\item $C(X)$ is a semiprojective $C^*$-algebra.
	\item $C(X)$ is an $\ell$-open $C^*$-algebra.
	\item $X$ is an ANR and $\dim(X)\leq 1$.
\end{enumerate}
\end{thm}
\begin{proof}
(i)$\Rightarrow$(ii) follows from \cite[Corollary 6.2]{Bla16} and (iii)$\Rightarrow$(i) follows from \cite[Theorem 1.2]{ST12}.

(ii) $\Rightarrow$ (iii) Suppose $C(X)$ is $\ell$-open. Blackadar showed that $X$ is $e$-open and thus locally contractible \cite[Corollary 4.3]{Bla12}.
The Homotopy Lifting Theorem (Theorem~\ref{thm:HomotopyLifting}) implies the homotopy extension theorem for $X$; since $X$ is also locally contractible, we have that $X$ is an ANR by \cite[Theorem IV.2.4]{Hu65}. We conclude that $dim(X) \le 1$ using Proposition 5.3.
\end{proof}

% Let $\phi_n : A \rightarrow B/I$ be a sequence of liftable $^*$-ho\-mo\-mor\-phisms which converges point-norm to a $^*$-ho\-mo\-mor\-phism $\phi : A \rightarrow B/I$. For $\epsilon =1$, choose $\delta$ as in Theorem 1.1. Then, there exists $k\in \mathbb(N)$ such that $\phi_k (a_j)
%\end{proof}

\
\end{document}